\documentclass[10pt,a4paper]{amsart}
\usepackage{amsthm,amsmath}
\usepackage{fourier}
\usepackage{nicefrac}
\usepackage{ulem}
\usepackage[cal=rsfso,calscaled=.96]{mathalfa}
\usepackage[colorlinks]{hyperref}
\hypersetup{
  pdftitle   = {},
  pdfauthor  = {},
  pdfcreator = {\LaTeX\ with package \flqq hyperref\frqq}
}
\usepackage{tikz}

\newtheorem{theorem}{Theorem}[section]
\newtheorem*{theorem*}{Theorem}

\newtheorem{lemma}[theorem]{Lemma}
\newtheorem*{lemma*}{Lemma}
\newtheorem{corollary}[theorem]{Corollary}

\newtheorem*{conjecture*}{Conjecture}

\theoremstyle{definition}
\newtheorem{definition}[theorem]{Definition}

\newtheorem{notation}[theorem]{Notation}
\theoremstyle{remark}
\newtheorem{remark}[theorem]{Remark}

\DeclareMathOperator{\Ext}{Ext}

\DeclareMathOperator{\Hom}{Hom}
\DeclareMathOperator{\Pic}{Pic}
\DeclareMathOperator{\Tot}{Tot}
\DeclareMathOperator{\Coh}{Coh}

\DeclareMathOperator{\tr}{tr}
\DeclareMathOperator{\im}{im}
\DeclareMathOperator{\diag}{diag}
\DeclareMathOperator{\Fuk}{Fuk}
\DeclareMathOperator{\SL}{SL}
\DeclareMathOperator{\GL}{GL}
\DeclareMathOperator{\Ad}{Ad}

\DeclareMathOperator{\id}{id}
\DeclareMathOperator{\ind}{ind}

\numberwithin{equation}{section}

\begin{document}
\title[A Landau--Ginzburg model without projective mirrors] {A Lie theoretical construction of a Landau--Ginzburg model without projective mirrors}

\author{E. Ballico, S. Barmeier, E. Gasparim, L. Grama, L. A. B. San Martin}
\address{Ballico --- Department of Mathematics, University of Trento, I--38050 Povo, Trento, Italy}
\address{Barmeier --- Westf\"alische Wilhelms-Universit\"at M\"unster, Mathematisches Institut, Einsteinstr.\ 62, M\"unster, Germany}
\address{Gasparim --- Departamento de Matem\'aticas, Universidad Cat\'olica del Norte, 
Av.\ Angamos 0600, Antofagasta, Chile}
\address{Grama and San Martin --- Imecc -- Unicamp, Departamento de Matem\'{a}tica, Rua S\'{e}rgio Buarque de Holanda,
651, Cidade Universit\'{a}ria Zeferino Vaz, 13083--859 Campinas -- SP, Brasil}
\address{E-mails: {\rm edoardo.ballico@unitn.it, s.barmeier@gmail.com, etgasparim@gmail.com, linograma@gmail.com, smartin@ime.unicamp.br.}}

\begin{abstract}
We describe the Fukaya--Seidel category of a Landau--Ginzburg model $\mathrm{LG}(2)$ for the semisimple adjoint orbit of $\mathfrak{sl}(2, \mathbb C)$. We prove that this category is equivalent to a full triangulated subcategory of the category of coherent sheaves on the second Hirzebruch surface. 
We show that no projective variety can be mirror to $\mathrm{LG}(2)$, and that this remains so after compactification. 
\end{abstract}

\maketitle
\tableofcontents

\section{Introduction}
We describe the Fukaya--Seidel category corresponding to a Landau--Ginzburg model for the semisimple adjoint orbit of $\mathfrak{sl} (2, \mathbb C)$. This is the simplest application of the following general result:

\begin{theorem} \cite[Thm.\,3.1]{GGSM1}
\label{thmslf} Let $\mathfrak h$ be the
Cartan subalgebra of a complex semisimple Lie algebra $\mathfrak g$. Given $H_0 \in 
\mathfrak h$ and $H \in \mathfrak h_{\mathbb R}$ with $H$ a regular
element, the \textit{height function} $f_H \colon \mathcal O (H_0) \to \mathbb C$ defined by 
\begin{equation*}
f_{H} (x) = \langle H, x \rangle, \qquad x \in \mathcal O (H_0) 
\end{equation*}
has a finite number (${}= \lvert \mathcal{W} \rvert / \lvert \mathcal W_{H_0} \rvert$) of isolated
singularities and gives $\mathcal O (H_{0})$ the structure of a
symplectic Lefschetz fibration.
\end{theorem}

Here $\mathcal O(H_0)$ denotes the adjoint orbit of $H_0$ viewed as a symplectic submanifold
of $\mathfrak{sl} (2, \mathbb C)$ with the symplectic form 
\begin{equation}
\label{symp}
\Omega = \im \mathcal H,
\end{equation}
where $\mathcal H$ is the Hermitian form on $\mathfrak g$ defined by
\[
\mathcal H(u,v) = \langle u, J v \rangle,
\]
for $J$ any almost complex structure and $\langle \,\cdot\, {,} \,\cdot\, \rangle$ denoting the Cartan--Killing form. In the following we will take $J$ to be multiplication by $i$ coordinatewise.

In the language of mirror symmetry, $f_H$ is called a superpotential.

\begin{notation}
\label{orbit}
Let us denote by $\mathrm{LG} (2)$ the Landau--Ginzburg model formed by the pair $(X, f_H)$ where 
$X := \mathcal O (H_0)$ is the semisimple orbit of $\mathfrak{sl} (2, \mathbb C)$ for
$H_0 = \left(
\begin{smallmatrix}
1 & 0 \\
0 & -1
\end{smallmatrix}
\right)$
considered as a symplectic manifold with the symplectic form as in (\ref{symp}) and given the structure of a symplectic Lefschetz fibration by the superpotential $f_H \colon X \rightarrow \mathbb C$ for the choice $H = H_0$. 
\end{notation}

We calculate the category of Lagrangian vanishing cycles for $\mathrm{LG}(2)$ and obtain:

\vspace{5pt plus 1pt minus 1pt}

\noindent {\bf Theorem \ref{teosl2}.} \, {\it
The Fukaya--Seidel category $\Fuk(\mathrm{LG}(2))$ 
is generated by two Lagrangians $L_0$ and $L_1$ with morphisms: 
\begin{align*}
\Hom (L_i, L_j) 
\simeq
\begin{cases}
\mathbb Z \oplus \mathbb Z [-1] & i < j \\
\mathbb Z                       & i = j \\
0                               & i > j
\end{cases}
\end{align*}
and the products $m_k$ all vanish except for $m_2(\,\cdot\,, \id)$ and $m_2({\id,}\,\cdot\,)$.}

\vspace{5pt plus 1pt minus 1pt}

We then consider the question of finding a mirror to $\mathrm{LG}(2)$. That is, we look for an algebraic variety $Y$ 
such that its derived category of coherent sheaves $D^b(\Coh Y)$ is equivalent to the Fukaya--Seidel category of $\mathrm{LG}(2)$. 
We first obtain a negative result. 

\vspace{5pt plus 1pt minus 1pt}

\noindent {\bf Theorem \ref{nopmirror}.} \, {\it $\mathrm{LG}(2)$ has no projective mirrors.}

\vspace{5pt plus 1pt minus 1pt}

This came to us as a surprise and brought along the question of whether the absence of projective mirrors might have resulted of the noncompactness of $\mathrm{LG}(2)$. We then compactified $\mathrm{LG}(2)$ to a new model $\overline{\mathrm{LG}}(2)$ where we extend the potential to a map with target $\mathbb P^1$. 
However, for the compactified $\overline{\mathrm{LG}}(2)$, the absence of projective mirrors persists:

\vspace{5pt plus 1pt minus 1pt}
\noindent {\bf Theorem \ref{cnopmirror}.} \, {\it $\overline{\mathrm{LG}}(2)$ has no projective mirrors.}
\vspace{5pt plus 1pt minus 1pt}

The next best thing to do then is to find some projective variety $Y$ such that a proper subcategory of $D^b(\Coh Y)$ is equivalent to $\Fuk(\mathrm{LG}(2))$.
We find that an appropriate choice is $Y= F_2$, the second Hirzebruch surface.

\vspace{5pt plus 1pt minus 1pt}

\noindent {\bf Theorem \ref{sigma2}.} \,
{\it $\Fuk (\mathrm{LG} (2))$ is equivalent to the full triangulated subcategory $D^b (\reflectbox{$\mathrm{LG}$} (2)) := \langle \mathcal O_{F_2}, \mathcal O_{F_2} (-E) \rangle $ of $ D^b (\Coh F_2)$, where $ F_2$ is the second Hirzebruch surface and $E$ is the divisor with self-intersection $-2$.}
\vspace{5pt plus 1pt minus 1pt}

We also describe these categories using quivers in \S\ref{quivers}. 

\begin{remark}
\label{remarkkhs}
It turns out that for the choices made in \ref{orbit}, we obtain an example already studied by Khovanov and Seidel in \cite{KhS}, where they describe the Fukaya category of the Milnor fibration corresponding to an $A_m$ singularity. More precisely, they consider deformations of the $A_m$ singularities, and their case $m = 1$ happens to be algebraically and symplectically isomorphic to the adjoint orbit $\mathcal O (H_0)$ of $\mathfrak{sl} (2, \mathbb C)$, see \S\ref{structures}.
\end{remark}

In our approach outlined above, we use Lie theory to define a potential on $\mathcal O (H_0)$, making it into a Landau--Ginzburg model, and then obtain information about the mirror category.

In future work we intend to consider the cases of adjoint orbits of the Lie algebras $\mathfrak{sl} (n, \mathbb C)$ with $n>2$. Then, the corresponding spaces will not be deformations of $A_m$ singularities of Remark \ref{remarkkhs}, since such adjoint orbits have dimension strictly greater than 2.

We hope that illustrating this simplest case using an alternative or complementary approach will lead to further study of symplectic Lefschetz fibrations and their Fukaya--Seidel categories using techniques from Lie theory. In light of Theorem \ref{thmslf}, this approach can indeed be formulated for all semisimple Lie algebras and we hope that it will lead to further results in mirror symmetry.

\vspace{5pt plus 1pt minus 1pt}

\paragraph{\it Acknowledgements.} \, We are grateful to Patrick Clarke for pointing out 
a significant improvement to an earlier version of this work. We thank Denis Auroux, Lutz Hille, Ludmil Katzarkov, and Sukhendu Mehrotra for helpful suggestions and comments. 

E.\ Ballico was partially supported by MIUR and GNSAGA of INdAM (Italy). S.~Barmeier is supported by the Studienstiftung des deutschen Volkes. E.\ Gasparim was partially supported by a Simons Associateship ICTP, and Network Grant NT8, Office of External Activities, ICTP, Italy. Part of this work was completed during a visit of L.~Grama to Chile. We are thankful to the Vice Rector\'ia de Investigaci\'on and Desarrollo Tecnol\'ogico of the Universidad Cat\'olica del Norte whose support made this visit possible. L.~Grama is partially supported by FAPESP grant 2016/22755-1.

\section{The structures of $\mathrm{LG}(2)$}
\label{structures} 

In this section we describe the Landau--Ginzburg model $\mathrm{LG}(2) = (\mathcal O (H_0), f_H)$ 
defined in \ref{orbit} corresponding to the choices $H_0 = H = \left(\begin{smallmatrix}1 & 0 \cr 0 & -1 \end{smallmatrix}\right)$ and symplectic form $\Omega(A,B) = \im \langle A, i B \rangle$.

\vspace{5pt plus 1pt minus 1pt}

\noindent{\sc Algebraic structure.}
Set $X := \mathcal O (H_0)$. Given
$A = 
\left(\begin{smallmatrix}
x &  y \\
z & -x
\end{smallmatrix}\right)
\in X \subset \mathfrak{sl} (2, \mathbb C)$, its eigenvalues are $\pm 1$ and
\[
(x - \lambda) (-x - \lambda) - yz = \det (A - \lambda I) = (\lambda + 1) (\lambda - 1) = \lambda^2 - 1.
\]
Hence, $X$ is the hypersurface in $\mathbb C^3$ cut out by the equation 
\begin{equation}
\label{eq:sl2orbit}
x^2 + y z - 1 = 0.
\end{equation}
Since the derivatives of the polynomial $x^2+yz-1$ vanish simultaneously only at the origin which does not lie in $X$, it follows that $X$ is a smooth complex surface.

We know that in the case of $\mathfrak{sl} (n, \mathbb C)$, $\langle A, B \rangle$ is a constant multiple of $\tr(AB)$. The choice of $H = \left(\begin{smallmatrix}1 & 0 \cr 0 & -1 \end{smallmatrix}\right)$ gives the height function 
\[
  f_H (A)
  =
  \tr (HA)
  =
  \tr
  \begin{pmatrix}
    1 & 0 \\
    0 & -1
  \end{pmatrix}
  \begin{pmatrix}
    x & y \\
    z & -x
  \end{pmatrix}
  =
  2x.
\]
So, we write the potential as
\begin{equation}\label{potential}
\begin{aligned}
f_H \colon X &\rightarrow \mathbb{C} \\
(x,y,z) &\mapsto 2x.
\end{aligned}
\end{equation}

\vspace{5pt plus 1pt minus 1pt}

\noindent{\sc Smooth structure.}
$X$ is not compact. In further generality, let $\mathfrak{u}$ be a real compact form of $\mathfrak{sl}(n,\mathbb{C})$, then
 \cite[Thm.\,2.1]{GGSM2} proves that the semisimple adjoint orbit is 
diffeomorphic to the cotangent bundle of the generalized flag variety $\mathcal{O}(H_0)\cap i\mathfrak{u}$.
For the orbit of $\mathfrak{sl}(2,\mathbb{C})$ the flag variety
is $\mathbb{P}^1 \approx S^2$ and consequently we have the diffeomorphism $X \simeq T^*S^2$.

\vspace{5pt plus 1pt minus 1pt}

\noindent{\sc Complex structure.}
Let $Z_2 = \Tot \mathcal O_{\mathbb P^1} (-2) = T^* \mathbb P^1$ with its canonical complex structure, and let $\tau \in H^1 (Z_2, T Z_2)$ be a non-zero cohomology class. Denote by $\mathcal Z_2 (\tau)$ the complex deformation of $Z_2$ corresponding to $\tau$, see \cite[\S 4]{BG} or \cite{Ba} for details. Observe that $Z_2$ is not an affine variety (as the nontrivial first cohomology shows), hence the complex structure of $X$ cannot be isomorphic to that of $Z_2$. We claim that $X$ is biholomorphic to $\mathcal Z_2 (\tau)$. In fact, the algebra of global functions of $\mathcal Z_2 (\tau)$ can be calculated  via \v Cech cohomology using canonical transition functions as in \cite{BG}, giving:\footnote{Details of the cohomology calculations are presented in  \cite{Ba}. Global functions on $\mathcal Z_2 (\tau)$ were also described  by Tyurina in  \cite{Tyu}  in the context of  resolutions  of the  $A_1$ singularity.
 The space $\mathcal Z_2 (\tau)$ had been studied earlier  by Atiyah in \cite{At}.}
 \[
\mathbb C [x, y, z] \big/ \big( (x + 1)^2 - y z - 1 \big).
\]
The change of coordinates $(x, y) \mapsto (x-1, -y)$ 
shows that $\mathcal Z_2 (\tau) \simeq X$ as affine varieties.

\vspace{5pt plus 1pt minus 1pt}

\noindent{\sc Symplectic structure.}
We have just shown that the diffeomorphism type of $X$ is that of the cotangent bundle of a sphere.
The next result shows that this sphere is a Lagrangian subvariety of $X$.

\begin{lemma}\label{lag}
Consider the orbit $X$ with the symplectic form $\Omega$ defined in
(\ref{symp}), then $Y \subset X$ given by the real equation $p^2 + q^2 + r^2 = 1$
is a Lagrangian submanifold.
\end{lemma}

\begin{proof}
Let $\mathfrak{u}$ be a real compact form of $\mathfrak{sl}(2,\mathbb{C})$.
Here $\mathfrak{u}$ is the set of anti-Hermitian matrices with trace zero,
thus $i\mathfrak{u} $ is the set of Hermitian matrices with trace
zero. Note that the submanifold $Y$ can be described as the intersection $Y = X \cap i \mathfrak{u}$. In fact, an arbitrary matrix $S \in i \mathfrak{u}$ has the form 
\begin{equation*}
S=
\begin{pmatrix}
r & -p+iq \\ 
-p-iq & -r
\end{pmatrix}, 
\end{equation*}
with $p,q,r \in \mathbb{R}$. Since the orbit $X$ consists of $2 \times 2$
complex matrices whose entries satisfy $x^2+yz=1$, we see that $S \in X$
if and only if its entries satisfy $p^2+q^2+r^2=1$.

The tangent space of $Y$ at $S$ is given by $T_S Y = \{ [S,A] \mid A \in \mathfrak{u} \} $. Since $[i\mathfrak{u},\mathfrak{u}] \subset i\mathfrak{u}$ and $\tr(MN)$ is real when $M,N \in i \mathfrak{u}$, we conclude that $\Omega_S ([S,A],[S,B]) = 0$; thus $Y$ is Lagrangian.
\end{proof}

\begin{remark}
In greater generality, let 
$\mathfrak{u}$ be a real compact form of $\mathfrak{g}$. 
The intersection $\mathcal{O} (H_0) \cap i \mathfrak{u}$ is a generalized flag variety, 
and a similar argument shows that such a generalized flag
variety is Lagrangian for the symplectic form $\Omega$.
\end{remark}

\section{The Fukaya--Seidel category of $\mathrm{LG}(2)$} 
In this section we prove:

\begin{theorem}
\label{teosl2} The Fukaya--Seidel category of $\mathrm{LG}(2)$
is generated by two Lagrangians $L_0$ and $L_1$ with morphisms: 
\begin{align}
\label{morphismsfukaya}
\Hom (L_i, L_j) 
\simeq
\begin{cases}
\mathbb Z \oplus \mathbb Z [-1] & i < j \\
\mathbb Z                       & i = j \\
0                               & i > j
\end{cases}
\end{align}
where we think of $\mathbb Z$ as a complex concentrated in degree $0$ and $\mathbb Z [-1]$ as its shift, concentrated in degree $1$, and the products $m_k$ all vanish except for $m_2(\,\cdot\,, \id)$ and $m_2({\id,}\,\cdot\,)$.
\end{theorem}

We will now describe the thimbles using branched covers. As described in \S \ref{structures} the orbit is 
$X= \{x^2+yz=1\}$ together with the potential 
\begin{align*}
f_H \colon X &\rightarrow \mathbb{C} \\
(x,y,z) &\mapsto 2x.
\end{align*}

For each regular value $c \in \mathbb{C}$ we have $f_H(A)=
2x=c$ 
and a corresponding regular fibre over $c$, to simplify notation we parametrize the regular fibres 
by $\lambda:=c/2$, so 
\begin{equation*}
X_\lambda := \left\{yz= 1 - \lambda^2
\right\}.
\end{equation*}

From the above description it is immediate that the singular fibres occur when $\lambda^2 = 1$. The singular fibres $X_{\pm 1} = f_H^{-1} (\pm 1) = \{ yz = 0 \}$ correspond to the critical points $(x,y,z) = (\pm 1,0,0)$ of the potential $f_H$.

We first consider the cut given by $y=z$ where we need to analyse the two
branches of the square root $y = \pm \sqrt{1 -\lambda^2}.$ We get
the two curves 
\begin{equation*}
\left(\lambda,\pm \sqrt{1 -\lambda^2},\pm \sqrt{1 - 
\lambda^2}\right) \overset{\lambda \rightarrow 1}{\longrightarrow}
(1,0,0).
\end{equation*}
Using these curves we want to write down the thimbles, that is, for each $\lambda$ we wish to identify a circle in $X$ parametrized by $\gamma(t)$ with
$\gamma (0)   = \big( \lambda,  \sqrt{1 - \lambda^2},  \sqrt{1 - \lambda^2} \big)$ and
$\gamma (\pi) = \big( \lambda, -\sqrt{1 - \lambda^2}, -\sqrt{1 - \lambda^2} \big)$. For 
$0 \leq t \leq 2\pi$ we chose the thimble as: 
\begin{equation*}
\alpha_\lambda(t)= \left(\lambda , e^{it} \sqrt{1 - \lambda^2}, e^{-it} \sqrt{1 -\lambda^2}\right).
\end{equation*}
Thus, $\alpha_\lambda(t) \rightarrow (1,0,0) $ as $\lambda \rightarrow 1$ (so that $c \rightarrow 2$)
and for a regular value $\lambda$ the curve $\gamma(t) := \alpha_\lambda(t) $
is a Lagrangian circle on the fibre $f_H^{-1}(2\lambda).$ We fix the regular
value $0\in \mathbb{C}$, and consider the straight line joining $0$ to the critical value $2$; this is our 
choice of a {\it matching path}. Then
the family of Lagrangian circles $\alpha_\lambda (t)$ is fibred over this
matching path and produces the Lagrangian thimble. With a similar analysis
we can produce the Lefschetz thimble associated to the critical value $-2$.

Consider now the thimbles over the union of the two matching paths (line joining the two critical values $-2$ and $2$), the circles fibering over them result in a 
sphere $Y$ inside the orbit $X$. As shown in Lemma \ref{lag} this sphere is Lagrangian in $X$.

We will now describe the Fukaya--Seidel category $\Fuk (\mathrm{LG} (2))$ associated to the
Landau--Ginzburg model $\mathrm{LG} (2)$, which is the category of vanishing cycles defined as follows.

\begin{definition}
\cite[Def.\,3.1]{AKO1} A \textit{directed category of vanishing cycles} associated to a Landau--Ginzburg model is an $A_\infty$-category (over a coefficient ring $R$) with $r$ objects $L_1, \dotsc, L_r$ corresponding to the vanishing cycles (or more accurately, to the thimbles); the morphisms between the objects are given by 
\begin{equation}
\Hom (L_i, L_j) =
\begin{cases}
\mathit{CF}^* (L_i, L_j) = R^{\lvert L_i \cap L_j \rvert} & \text{if } i<j \\
                           R \cdot \id                    & \text{if } i=j \\
                           0                              & \text{if } i>j
\end{cases}
\end{equation}
and the differential $m_1$, composition $m_2$ and higher order products $m_k$ are defined in terms of Lagrangian Floer homology inside the regular fibre. See \cite{AKO1} for further details.
\end{definition}

We fix the regular value $0 \in \mathbb{C}$ of
our Landau--Ginzburg model and consider the line segments  $\beta$ and $\gamma$ that join $-2$ to $0$ and $0$ to $2$, respectively. The objects of the Fukaya--Seidel category are the two Lagrangian thimbles  $L_0:=\alpha_{\beta(s)}(t)$ and $L_1:=\alpha_{\gamma(s)}(t)$ (abusing notation we consider as $L_0$ and $L_1$ only the vanishing cycles in the regular fibre $X_0$; in our case, both  circles $S^1$).

Note that the vanishing cycles represent a single object in the Fukaya category of the regular fibre, but represent two distinct objects in the Fukaya--Seidel category of $\mathrm{LG} (2)$.

To specify the products in the category, we need to describe $\mathit{CF}^* (L_0, L_1)$. The regular fibre $X_0$ is homeomorphic to $\mathbb{C}^*$ and to the cylinder $T^*S^1$ via the map $g \colon \mathbb{C}^* \rightarrow T^*S^1$ given by 
\begin{equation*}
g(y) = \bigg( \frac{y}{\lvert y \rvert}, \ln \lvert y \rvert \bigg).
\end{equation*}
In the regular fibre the vanishing cycles can be parametrized by the
curve $(0,e^{it},e^{-it}) \in X_0$ by setting $\lambda=0$ in the 
expressions for the thimbles. Moreover, Lemma \ref{lag} 
implies that $L_0$ (and $L_1$) is Lagrangian in $X_0$ 
and therefore by Weinstein's theorem we have that a 
tubular neighbourhood of $L_0$ is symplectomorphic to the cotangent
bundle $T^*S^1$. 
In this situation the Floer homology is well known,
see \cite{Au} and \cite{FOOO}.

\begin{lemma}
\label{floer-iso}
$\mathit{HF}^*(L_0,L_1)\approx H^*(S^1;\mathbb{R})$. 
\end{lemma}

We now fix a Morse function $f \colon S^1 \rightarrow \mathbb{R}$ with exactly two critical points. A critical point $p$ of $f$ with Morse index $\ind(p)$ defines a generator of degree $\deg(p)=n-\ind(p)$ in the Floer complex, where $n$ is the dimension of the variety (in our case $\dim S^1 = 1$). Since we have chosen $f$ with exactly two critical points, a minimum $x_0$ and a maximum $x_1$, the Morse indices are $0$ and $1$, respectively. We obtain:

\begin{lemma}
There is a natural choice of grading such that $\deg(x_0) = 0$ and $\deg(x_1) = 1$.
\end{lemma}

Since the product $m_1$ in the Fukaya--Seidel category is
the differential of Floer homology, using Lemma \ref{floer-iso}, we obtain
the following description of the products $m_k$:

\begin{lemma}
\label{productsvanish}
The products $m_k$ for the Fukaya--Seidel category of $\mathrm{LG} (2)$ all vanish,
except for the trivial products $m_2 ({\id,}\,\cdot\,)$ and $m_2 (\,\cdot\,,\id)$.
\end{lemma}

Here, the strict unit $\id$ equals $x_0$, and the result follows from
strict unitality and the degree considerations. Specifically,
$m_2 (x_1, x_1)$ has degree $2$ and so it is zero.
Moreover, strict unitality implies that the only possible non-zero products for $k > 2$
take only $x_1$ as argument, and $m_k (x_1, \dotsc, x_1)$ is zero because it has
degree $2 - k + k = 2$.

\begin{remark}
We compare with the mirror of $\mathbb{P}^1$. 
The Fukaya--Seidel
category we just described is not isomorphic to the Fukaya--Seidel category
of the mirror of $\mathbb{P}^1$ described in \cite{AKO1}. Indeed, although
the number of objects, morphisms and products of the $A_\infty$ structures
coincide, the gradings are different, hence the categories are not equivalent. 
We give a more detailed argument for this in the proof of Theorem \ref{nopmirror}.
\end{remark}

\section{Mirror candidates}\label{mirror}

We show that no projective variety is mirror to $\mathrm{LG}(2)$.
In other words, suppose that we have a variety $Y$ such that the bounded derived
 category $D^b (\Coh Y)$ of
 coherent sheaves on $Y$ is equivalent to our Fukaya--Seidel category of Theorem \ref{teosl2}. 
Thus, we would need to have that $D^b(\Coh Y)$ is generated by some $\mathcal F_0, \mathcal F_1 \in \Coh Y$ satisfying:
\begin{align*}
\Hom (\mathcal F_i, \mathcal F_j) 
\simeq
\begin{cases}
\mathbb C \oplus \mathbb C [-1] & i < j \\
\mathbb C                       & i = j \\
0                               & i > j.
\end{cases}
\end{align*}
We prove that any such variety $Y$ cannot be projective. Hence:

\begin{theorem}
\label{nopmirror}
$\mathrm{LG}(2)$ has no projective mirrors.
\end{theorem}

\begin{proof}
We first argue that if $\dim Y = n > 1$ and $Y$ is projective, then $D^b (\Coh Y)$ cannot be generated by two simple objects $\mathcal L_0, \mathcal L_1$ such that $\Hom (\mathcal L_i, \mathcal L_i) = \mathbb C$ for $i = 0, 1$.

We will use the following facts. First, if $\mathcal C$ is an abelian category, such as $\Coh Y$ for any scheme $Y$, then the Grothendieck group $K (\mathcal C)$ of $\mathcal C$ is isomorphic to the Grothendieck group $K (D^b (\mathcal C))$ of the bounded derived category of $\mathcal C$. Recall that the Grothendieck group in either case is generated by the isomophism classes of objects in the respective category. The relations in the first case are given by short exact sequences\footnote{namely, if $0 \to A \to B \to C \to 0$ is a short exact sequence, then $[B]=[A]+[C]$ in $K(\mathcal C)$}, while in the latter by exact triangles, see \cite[Ex.\,1.27]{KS}. 

Second, if $\langle \mathcal A, \mathcal B \rangle$ is a semi-orthogonal decomposition of a triangulated category $\mathcal D$, {\it e.g.}\ $\mathcal D = D^b (\Coh Y)$, then $K (\mathcal D) = K (\mathcal A) \oplus K (\mathcal B)$. Note that the Grothendieck group of a triangulated category is defined in the obvious way: the generators are the isomorphism classes of objects, the relations come from exact triangles. 

Last, if $D^b (\Coh Y)$ admits a semi-orthogonal decomposition by sheaves $\mathcal F_1, \dotsc, \mathcal F_m$ together with another factor $\mathcal A$, that is,
\[
D^b (\Coh Y) = \langle \mathcal F_1, \mathcal F_2, \dotsc, \mathcal F_m, \mathcal A \rangle,
\]
then
\[
G_0 (Y) := K (\Coh Y) = K (D^b (\Coh Y)) = K (\mathcal F_1) \oplus \dotsb \oplus K (\mathcal F_m) \oplus K (\mathcal A),
\] 
where by $K (\mathcal F_i)$ we mean the Grothendieck group of the full triangulated category generated by $\mathcal F_i$, each of which is isomorphic to $\mathbb{C}$. (We assume $Y$ is a scheme over the complex numbers, but this works over any field.) Thus, $\dim G_0 (Y) \geq m$, as claimed (or use \cite[Prop.\,2.1]{wei}). Since $\dim G_0 (Y) \geq n+1$, we get $n=1$.

Assume now that $n = 1$. If the normalization $Y'$ of $Y$ has geometric genus $\ge 1$, then \cite[Prop.\,4.6]{wei} gives that $G_0 (Y)$ is not finitely generated.
If $Y' = \mathbb {P}^1$ and $Y \ne \mathbb {P}^1$, then \cite[Prop.\,4.1]{wei} gives that a categorical resolution (in the sense of \cite{kl}) of $D^b (\Coh Y)$ has a full exceptional collection, but the proof of \cite[Prop.\,4.1]{wei} gives that its length $m$ is at least $3$. Hence $G_0 (Y) \ne \mathbb Z^2$.

Finally, we exclude the case $Y = \mathbb {P}^1$. Assume $Y = \mathbb {P}^1$ and that $\mathcal L_0$ and $\mathcal L_1$ are simple objects of $D^b (\Coh \mathbb {P}^1)$. Since $\mathbb {P}^1$ is a smooth curve, every coherent sheaf $\mathcal F$
on $\mathbb {P}^1$ is a direct sum of a torsion sheaf $\mathrm{Tors}(\mathcal F)$ and a locally free sheaf $\mathcal F\! \big/ \mathrm{Tors}(\mathcal F)$. Every locally free sheaf is isomorphic to a direct sum of line bundles.
Every torsion sheaf is a direct sum of skyscraper sheaves $\mathcal {O}_p$, $p \in \mathbb {P}^1$. Hence the only simple coherent sheaves on $\mathbb {P}^1$ are the line bundles $\mathcal {O} _{\mathbb {P}^1}(t)$,
$t\in \mathbb {Z}$, and the sheaves $\mathcal {O}_p$, $p\in \mathbb {P}^1$. No pair of them, not even after a shift $\mathcal L_0 [-i]$, $\mathcal L_1 [-j]$ may be of this form:
if $p, q\in \mathbb {P}^1$ and $p\ne q$, then $\Ext^i (\mathcal O_p, \mathcal O_q) = 0$ for all $i$, either $\Hom (\mathcal R, \mathcal L) = 0$ or $\Ext^1 (\mathcal L, \mathcal R)=0$ for any line bundles $\mathcal L, \mathcal R$,
$\Hom (\mathcal {O}_p, \mathcal L)= 0$ and $\dim \Ext^1(\mathcal {O} _p, \mathcal L) =1$ for all $p \in \mathbb {P}^1$ and any line bundle $\mathcal L$. Since $\mathbb {P}^1$ is a smooth curve,
\cite[Prop.\,6.3]{gkr} gives that every simple element of $D^b (\mathbb {P}^1)$ is isomorphic to some $\mathcal F [-i]$ with $\mathcal F$ a simple coherent sheaf on $\mathbb {P}^1$.
\end{proof}

We now proceed to the task of compactifying our Landau--Ginzburg model and verifying the 
effect of compactification on the Fukaya--Seidel category. 

\section{Compactification of the orbit}\label{comp}

Recall that the orbit $X$ is an affine surface in $\mathbb C^3$, as described in (\ref{eq:sl2orbit}). We will 
embed it into a projective surface $\overline X$, and see that the natural choice 
is $\overline X = \mathbb P^1 \times \mathbb P^1$. 
We compactify $X$ by homogenizing equation (\ref{eq:sl2orbit}).
This produces the projective surface $\overline X$ cut out by 
$x^2+yz-t^2=0$ in $\mathbb P^3$, that can be taken to the standard quadric equation by 
 the change of coordinates $x \mapsto x-t$ and $t \mapsto x+t$, 
hence the surface is $\mathbb P^1 \times \mathbb P^1$. 
This compactification also works well from the symplectic point of view. Thus, we have:

\begin{theorem} \label{symplectic}
The semisimple adjoint orbit $(X, \Omega)$ of $\mathfrak{sl}(2,\mathbb C)$ compactifies holomorphically 
and symplectically to $\mathbb P^1 \times \mathbb P^1$.
\end{theorem}

 \begin{proof}

Recall from \S \ref{structures} that we may identify the complex structure of $X$ with that of a nontrivial deformation $\mathcal Z_2 (\tau)$ of $Z_2 = \Tot (\mathcal O_{\mathbb P^1} (-2))$. In fact, the deformation of $Z_2$ extends to a deformation of its natural compactification, the second Hirzebruch surface $F_2$ obtained from $Z_2$ by adding a line at infinity, an irreducible divisor with self-intersection $+2$. It is well known that the complex surface $F_2$ deforms to the Hirzebruch surface $F_0 \simeq \mathbb P^1\times \mathbb P^1$. Identifying $Z_2$ as a subset of $F_2$, this deformation corresponds to a nontrivial element $\tau \in H^1 (F_2, T F_2) \simeq H^1 (Z_2, T Z_2)$.

Under deformation of $F_2$, the added line at infinity decomposes into the sum $E + F$ of two divisors $E, F$ corresponding in the deformed surface $F_0 = \mathbb P^1 \times \mathbb P^1$ to the fibre and the zero section of $F_0$ (considered as the trivial $\mathbb P^1$-bundle over $\mathbb P^1$). The divisor $E + F$ is ample, and its complement is the affine variety $\mathcal Z_2 (\tau) \simeq X$.

Thus, the complex structure of $X$ agrees with the one inherited from $F_0$, and similarly the metric on $X$ agrees with the K\"ahler metric inherited from $F_0$. These together imply that there exists a unique compatible symplectic structure on $X$ fitting the compactification to $F_0$. On the other hand, it is clear from definition \ref{symp} that the symplectic structure $\Omega$ on $X$ is compatible with the complex structure on $\mathfrak{sl}(2,\mathbb C)$. Hence, the symplectic structure on $F_0$ restricts to $\Omega$ on $X$.
\end{proof}
 
Let us identify the compactified fibres of the Landau--Ginzburg model and the divisor at infinity. 
As seen in (\ref{potential}) the potential on the open orbit $X$ is $f_H (A) = 2x$ and it has critical values $\pm 2$. Thus, $0$ is a regular value, and we express the regular fibre over $0$, $X_0$, as the affine variety in $\{ (y,z) \in \mathbb C^2 \}$ cut out by the equation
\[
yz - 1 = 0
\]
since it must satisfy equation (\ref{eq:sl2orbit}) and $x = 0$. As with the orbit, we homogenize this equation and embed the fibre into the corresponding projective variety $\overline{X_0}$ cut out by the equations $x=0$ and $yz-t^2=0$ in $ \mathbb P^3$.  Here the complement of the orbit $\overline{X} \setminus X$ in the compactification is obtained by making $t = 0$, thus $x^2 - yz = 0$ inside a projective plane $\mathbb P^2$, hence a conic curve, that is, a $\mathbb P^1$.

Next we need to compactify the potential. We will first extend the potential as a rational map over $\overline {X}$ and this rational map will then give rise to a holomorphic map on a compactification $\overline\Gamma$. We shall choose the symplectic form on $\overline\Gamma$ such that it coincides with the original symplectic form on $X$ on an open neighborhood of its thimbles, thus keeping the Lagrangians
we used to build the Fukaya category. 

\section{The potential viewed as a rational map}\label{ratext}

Our goal now is to extend the potential to the compactification. We will make use of another incarnation of the orbit, namely the adjoint orbit of $e_1 \otimes \varepsilon_1$ in $\mathbb C^2 \otimes (\mathbb C^2)^*$. The various incarnations of the orbits are described for the general case in \cite[\S 4]{BGGSM}.
Here we will describe explicitly the isomorphism between two such incarnations for the case of $\mathfrak{sl}(2,\mathbb C)$, then we will use the tensor product version of the orbit to show that the compactification naturally induces the Segre embedding into $\mathbb P^3$. Our extension of the potential to a rational map on $\mathbb P^1 \times \mathbb P^1$ factors through the Segre embedding. Note that the potential does not extend to a holomorphic map, not even if we change the target to $\mathbb P^1$. In \cite[\S 6]{BGGSM} it is shown how to extend the potential to a rational map for the cases when the orbit is diffeomorphic to $T^*\mathbb P^n$; all other cases remain open. 

Let us first set up some notation. 
For this section we write $A\in \SL (2, \mathbb{C})$ as
\begin{equation}\label{sl2-grupo}
A= 
\begin{pmatrix}
x&z\\
y&w
\end{pmatrix},
\end{equation}
with $wx-yz=1$, and fix the following basis for the Lie algebra $\mathfrak{sl}(2,\mathbb{C})$:
\begin{equation}
\label{base-sl2}
H=
\begin{pmatrix}
1&0\\
0&-1
\end{pmatrix},
\qquad
X_\alpha=
\begin{pmatrix}
0&1\\
0&0
\end{pmatrix},
\qquad
X_{-\alpha}=
\begin{pmatrix}
0&0\\
1&0
\end{pmatrix}.
\end{equation}
We consider the representation of the group $\rho \colon \SL(2,\mathbb{C}) \rightarrow \GL(\mathbb{C}^2)$ by left multiplication
\[
\rho(A)v = Av
\]
and its dual representation $\rho^\ast\colon \SL(2,\mathbb{C}) \rightarrow \GL(\mathbb{C}^2)^\ast$ given by
\[
\rho^\ast (A) \varepsilon = \varepsilon \circ A^{-1}.
\]
We denote by $\theta := d_e \rho$ the corresponding representation of the Lie algebra $\mathfrak{sl} (2, \mathbb C)$.

Let $\alpha$ be the positive root of $\mathfrak{sl}(2,\mathbb{C})$, that is, $\alpha = \lambda_1 - \lambda_2$, where $\lambda_i$ is the functional $\lambda_i (\diag(x_1,x_2)) = x_i$, $i = 1, 2$.
The fundamental weight for $\theta\colon \mathfrak{sl}(2,\mathbb{C}) \rightarrow \mathfrak{gl} (\mathbb{C}^2)$ is
$\mu=\frac{1}{2}\alpha$, and the corresponding element in the Cartan subalgebra is 
\[
H_\mu=
\begin{pmatrix}
\frac12 &     0    \\
   0    & -\frac12
\end{pmatrix}.
\]
Consider the canonical basis $\{ e_1, e_2 \}$ of $\mathbb C^2$. The weight spaces of the representation $\theta$ are:
 $V_1=\mathrm{span}\{e_1\}$ and $V_{-1}=\mathrm{span}\{e_2\}$. Recall that $\theta(X_\alpha)$ maps $V_{-1}$ to $V_1$ 
 and that $\theta(X_{-\alpha})$ maps $V_1$ to $V_{-1}$. Explicitly, 
\begin{align*}
\theta(X_\alpha)\begin{pmatrix}
a \\
b
\end{pmatrix}
=
\begin{pmatrix}
b \\
0
\end{pmatrix}, \qquad
\theta(X_{-\alpha})
\begin{pmatrix}
a \\
b
\end{pmatrix}
=
\begin{pmatrix}
0 \\
a
\end{pmatrix}.
\end{align*}
We set $v_1 = (1,0) \in \mathbb{C}^2$ and $\varepsilon_1 = (1,0) \in (\mathbb{C}^2)^\ast$.

If $A\in \SL(2,\mathbb{C})$ is written as in (\ref{sl2-grupo}), then
\[
B = \Ad(A) H_\mu = A H_\mu A^{-1} =
\begin{pmatrix}
\frac{1}{2} (wx+yz) & -xz \\
yw &-\frac{1}{2} (wx+yz)
\end{pmatrix}.
\]
The eigenvectors of $B$ are $(x,y)$, associated to the eigenvalue $\frac{1}{2}$, and $(z,w)$, associated to the eigenvalue $-\frac{1}{2}$.

\begin{lemma}
The adjoint action on the tensor product expression of the orbit can be interpreted as the Segre embedding. 
\end{lemma}

\begin{proof}
We have the equality
\[
A \cdot (v_1 \otimes \varepsilon_1) = \rho (A) v_1 \otimes \rho^\ast (A) \varepsilon_1, 
\]
where
\[
\rho (A) v_1 =
\begin{pmatrix}
x & z \\
y & w
\end{pmatrix}
\begin{pmatrix}
1 \\
0
\end{pmatrix}
=
\begin{pmatrix}
x \\
y
\end{pmatrix}
\]
and
\[
\rho^\ast (A) \varepsilon_1 = \varepsilon \circ \rho (A^{-1}) =
\begin{pmatrix}
1 & 0
\end{pmatrix}
\begin{pmatrix}
 w & -z \\
-y &  x
\end{pmatrix}
=
\begin{pmatrix}
w & -z
\end{pmatrix}.
\]
Therefore, 
\begin{equation}
\label{tensor}
A \cdot (v_1 \otimes \varepsilon_1) =
\begin{pmatrix}
xw & -xz \\
yw & -yz
\end{pmatrix}.
\end{equation}
Note that the eigenvalues of (\ref{tensor}) are $0$ (with associated eigenvector $(z,w)$) 
and $1$ (with associated eigenvector $(x,y)$).

If we consider $(x,y)$ and $(z,w)$ as projective coordinates, then the action on the tensor product can be interpreted as
the {\it Segre embedding} of $\mathbb{P}^1 \times \mathbb{P}^1$ into $ \mathbb{P}^3$ (up to a sign), which is
 $([x:y],[z:w]) \mapsto [xz:xw:yz:yw]$.
\end{proof}

The next lemma provides a diffeomorphism between the orbit $\SL(2,\mathbb{C})\cdot (v_1 \otimes \varepsilon_1)$ 
and the adjoint orbit $\Ad(\SL(2,\mathbb{C})) H_\mu$.

\begin{lemma}\label{dif}
The orbit $\SL(2,\mathbb{C})\cdot (v_1 \otimes \varepsilon_1)$ is diffeomorphic to the adjoint orbit $\Ad(\SL(2,\mathbb{C})) H_\mu$.
\end{lemma}

\begin{proof}
The diffeomorphism between the orbits of $\SL(2,\mathbb{C)}$ will be written using the moment map
\begin{equation}
M (v \otimes \varepsilon) (Z) = \varepsilon (\theta (Z) v), 
\end{equation}
where $v \in \mathbb{C}^2, \varepsilon \in (\mathbb{C}^2)^\ast, Z \in \mathfrak{sl} (2,\mathbb{C})$.
Let $v = (x,y)$ and $\varepsilon = (z,w)$. To describe $M (v \otimes \varepsilon)$ in the base (\ref{base-sl2}), we write:
\begin{alignat*}{4}
\langle M(v\otimes\varepsilon)&, H \rangle &{}={}& \varepsilon(\theta(H)v) &{}={}& \varepsilon\big(\tfrac{1}{2}x, -\tfrac{1}{2}y\big)&{}={}&\tfrac{1}{2}(xw+yz) \\
\langle M(v\otimes\varepsilon)&, X_\alpha \rangle &{}={}& \varepsilon(\theta(X_{-\alpha})v) &{}={}& \varepsilon(0,x) &{}={}& {-xz} \\
\langle M(v\otimes\varepsilon)&, X_{-\alpha} \rangle &{}={}& \varepsilon(\theta(X_{\alpha})v) &{}={}& \varepsilon(y,0) &{}={}& yw \text{.}
\end{alignat*}
Therefore, 
\begin{equation}
M (v \otimes \varepsilon) =
\begin{pmatrix}
\frac{1}{2} (wx+yz)& -xz\\
yw & -\frac{1}{2}(wx+yz)
\end{pmatrix}
= \Ad(A) H_\mu.
\end{equation}
\end{proof}

\begin{theorem}\label{rational} The rational map $R_H\colon \overline{X} = \mathbb P^1\times \mathbb P^1 \rightarrow \mathbb P^1$  that extends the potential is 
\[
R_H ([x : y], [z : w]) = [xw + yz : xw - yz].
\]
\end{theorem}

\begin{proof}
Choosing $H = \diag (1, -1)$ we wish to extend the potential 
 $f_H$ to a rational map on the compactification \begin{equation} \label{compac}
R_H:\mathbb{P}^1\times \mathbb{P}^1 \dashrightarrow \mathbb{P}^1.
\end{equation}
The rational map $R_H$ that we are looking for is the map to $ \mathbb{P}^1$
associated to $f_H$, that is, the rational map defined on the compactification
 that coincides with $f_H$ in the open orbit. We claim that the extension is given by 
\begin{equation}\label{quotient}
R_H(v\otimes \varepsilon)= \frac{\tr((v\otimes \varepsilon) \theta(H))}{\tr(v\otimes \varepsilon)}=\frac{xw+yz}{xw-yz}.
\end{equation} 
Observe that:
\begin{itemize}
\item If $v \otimes \varepsilon$ belongs to the adjoint orbit, then $v \otimes \varepsilon$ has the form $A\cdot v_1 \otimes \epsilon_1$ for some $A \in \SL (2, \mathbb{C})$, that is, $v \otimes \varepsilon \in \mathbb{C}^2\otimes (\mathbb{C}^2)^\ast$ is a matrix of 
the form (\ref{tensor}).
\item The previous item implies that $\tr (v \otimes \varepsilon) = 1$ if $v \otimes \varepsilon$ are in the orbit. Therefore $R_H = f_H$ on the orbit.
\item The poles of $R_H$ are vectors whose coordinates satisfy $xw=yz$. In other words, $(x,y)$ is a multiple of $(z,w)$. 
These are the pairs that are {\it not} in the adjoint orbit (formed by transversal lines).
\end{itemize}
Therefore, the map defined by formula (\ref{quotient}) factors through the Segre embedding:
\begin{equation}\label{composta}
([x:y],[z:w]) \mapsto [xz: xw: yz: yw] \mapsto [xw+yz:xw-yz].
\end{equation}
and coincides with $f_H$ on the orbit, that is $$([x:y],[z:w])\mapsto [f_H:1].$$ 
$R_H$ is defined on points outside the orbit as $$([x:y],[z:w])\mapsto [2xw: 0],$$ 
{\it except} the points of the base locus $P_1=([1:0],[1:0])$ and $P_2=([0:1],[0:1])$, where the map is ill defined.
\end{proof}

\begin{remark}
The rational map in Theorem \ref{rational} is defined outside the points $P_1$ and $P_2$.
Observe that these points are associated to the nilpotent matrices:
\begin{equation*}
\begin{aligned}
([1:0],[1:0]) \mapsto [1:0:0:0] &\simeq
\begin{pmatrix}
0&1\\
0& 0
\end{pmatrix}
\\
([0:1],[0:1]) \mapsto [0:0:0:1] &\simeq
\begin{pmatrix}
0 & 0 \\
1 & 0
\end{pmatrix}.
\end{aligned}
\end{equation*}
\end{remark}

\section{The compactified $\mathrm{LG}$ model}

In Theorem \ref{rational} we extended the potential to a rational map $R_H\colon \overline{X} = \mathbb P^1\times \mathbb P^1 \rightarrow \mathbb P^1$ as 
\begin{equation}\label{rat}
([x:y],[z:w]) \mapsto [xw+yz:xw-yz].
\end{equation}
However, the map $R_H$ is ill defined at $P_1=([1:0],[1:0])$ and $P_2=([0:1],[0:1])$.
We wish to extend $R_H$ to a holomorphic map and will do so by blowing up.

\begin{notation}
We take coordinates $[r:s]$ on the target $\mathbb P^1$ and 
consider the graph $\Gamma$ of $R_H$ inside the product. We denote by
$\overline{\Gamma}$ the closure of $\Gamma$ in $\overline{X} \times \mathbb P^1$, 
hence $\overline{\Gamma}$ is the surface cut out inside 
$ \mathbb P^1 \times \mathbb P^1\times \mathbb P^1$ by 
$$s(xw+yz) = r(xw-yz).$$
\end{notation}

\begin{lemma}\label{seidel}
$\overline{\Gamma}$ is a holomorphic and symplectic compactification of $X$.
\end{lemma}

\begin{proof}
By construction $\overline{\Gamma}$ is a complex hypersurface 
of $\mathbb P^1 \times \mathbb P^1 \times \mathbb P^1$ obtained by blowing up 
points on $\overline{X}$. Hence is it clearly a holomorphic compactification of $X$. 
However, pulling back the symplectic form of $\overline{X}$ to $\overline{\Gamma}$
by the blow-up map gives rise to a form that is degenerate on the exceptional set. 
We will now fix the degeneracy. 

As shown in Theorem \ref{symplectic} the symplectic structure on $\overline X = \mathbb P^1\times \mathbb P^1$ is 
compatible with the one on $X$. In Theorem \ref{rational} the potential was extended to a rational map $R_H$ on $\overline X$. We need to adapt the symplectic structure on $\overline \Gamma$ to fit the situation. 
We claim that we have arrived at the situation of \cite[\S 3]{Se} where Seidel considers a holomorphic Morse function $\raisebox{.5ex}{$\sigma_0$}\big/\!\raisebox{-.25ex}{$\sigma_1$}$ defined on a smooth projective variety. In our case we have $\raisebox{.5ex}{$\sigma_0$}\big/\!\raisebox{-.25ex}{$\sigma_1$} = \raisebox{.5ex}{$xy+yz$}\big/\!\raisebox{-.25ex}{$xw-yz$}$ defined on
 $\mathbb P^1\times \mathbb P^1$.
In this situation, we then look at the Lefschetz fibration of hypersurfaces
\[
Y_z = \big\{ p \in X \mathbin{\big\vert} \raisebox{.5ex}{$\sigma_0 (p)$}\big/\!\raisebox{-.25ex}{$\sigma_1 (p)$} = z \big\}
\]
for $z \in \mathbb P^1 = \mathbb C \cup \{\infty\}$.
Note that here $Y_ \infty $ is smooth, as required by \cite{Se}. Thus, we arrived directly at the second stage of his construction, 
where we already have a Lefschetz fibration together with a rational function on it (without having passed by a Lefschetz pencil beforehand). Following his method of patching in a correction to the symplectic form on a small neighborhood of the exceptional set we then arrive at the desired symplectic form. For our purposes it is important to take the neighborhood 
small enough so that it does not intersect the thimbles we had in $X$, but this can be done since the points $P_1$ and $P_2$ where the $R_H$ was ill defined are far from the thimbles of $f_H$.
\end{proof}

We will now use the projection to $[r:s]$ to extend the rational map $R_H$ on $\overline{X}$ 
to a holomorphic map $F_H$ on $\overline{\Gamma}$.

\begin{theorem}
Let $\pi_3\colon \mathbb P^1 \times \mathbb P^1\times \mathbb P^1 \rightarrow \mathbb P^1$ be the projection onto the third factor 
and set $$F_H:= \pi_3\big\vert_{\overline{\Gamma}}.$$ 
Then $F_H $ is a holomorphic extension of $f_H$.
\end{theorem}

\begin{proof}
 In fact, for points in $\overline{\Gamma}$ 
we have that 
$$F_H([x:y],[z:w],[r:s] ) = [r:s] = [xw+yz:xw-yz ] = R_H([x:y],[z:w]).$$
Thus, $F_H$ is an extension of $R_H$ which in turn is an extension of $f_H$ as shown in Theorem \ref{rational}. 
\end{proof}

\begin{corollary}\label{critical}
The critical points of $F_H$ coincide with the critical points of $f_H$.
\end{corollary}

\begin{proof}
For a fixed value $[r_0:s_0]$ on the target $\mathbb P^1$, the fibre of $F_H$ is cut out inside $\mathbb P^1 \times \mathbb P^1$
by the polynomial equation
\[
s_0(xw+yz) - r_0(xw-yz) = 0.
\]
This describes a singular conic only in the cases when $s_0 = \pm r_0$, thus the only critical values of $F_H$ are $[1:1]$ and $[1:-1]$ with corresponding critical points $([1:0], [0:1])$ and $([0:1],[1:0])$. These in turn correspond to the critical points $\pm H$ of $f_H$. We conclude that extending $f_H$ to $F_H$ does not produce any extra critical points. 
\end{proof}

\begin{corollary}
The Fukaya--Seidel category of $\overline{\mathrm{LG}}(2)$ is the same as the one of ${\mathrm{LG}}(2)$.
\end{corollary}

\begin{proof}
Observe that in Lemma \ref{seidel} chose the symplectic form on the compactification $\overline \Gamma$ so that our original Lagrangian thimbles that generated $\Fuk(\mathrm{LG}(2))$ remain Lagrangian in the compactification. Moreover, Corollary \ref{critical} shows that no new critical points arise when we extend the potential to the compactification. Therefore the Fukaya--Seidel category corresponding to the compactification is the same as the one described in Theorem \ref{teosl2}.
\end{proof}

In particular, using the results of \S \ref{mirror} we conclude that this compact LG model $\overline{\mathrm{LG}}(2)$ does not have a projective mirror either. Hence we obtain:

\begin{theorem}\label{cnopmirror}
$\overline{\mathrm{LG}}(2)$ has no projective mirrors.
\end{theorem}

\section{Mirror category}

Theorem \ref{teosl2} states 
that the Fukaya--Seidel category of $\mathrm{LG} (2)$ is generated by two Lagrangians $L_0$ and $L_1$ with the following morphisms
\begin{align*}
\Hom (L_i, L_j) 
\simeq
\begin{cases}
\mathbb Z \oplus \mathbb Z [-1] & i < j \\
\mathbb Z                       & i = j \\
0                               & i > j
\end{cases}\text{.}
\end{align*}
Theorems \ref{nopmirror} and \ref{cnopmirror} show that no projective variety may be the mirror of either $\mathrm{LG} (2)$
or else $\overline{\mathrm{LG}} (2)$. However, we do have the following result, which in light of Lemma \ref{tilting} below may be thought of as an instance of \cite[Cor.\,2.7]{O}.

\begin{theorem}\label{sigma2}
$\Fuk (\mathrm{LG} (2))$ is equivalent to the full triangulated subcategory $D^b (\reflectbox{$\mathrm{LG}$} (2)) := \langle \mathcal O_{F_2}, \mathcal O_{F_2} (-E) \rangle $ of $ D^b (\Coh F_2)$, where $ F_2$ is the second Hirzebruch surface and $E$ is the divisor with self-intersection $-2$.
\end{theorem}

\begin{proof}
Let $[x_0 : x_1 : x_2]$ and $[y_0 : y_1]$ be the standard coordinates on $\mathbb P^2$ and $\mathbb P^1$. The second Hirzebruch surface is the hypersurface $F_2 \subset \mathbb P^2 \times \mathbb P^1$ cut out by the equation $x_0 y_0^2 - x_1 y_1^2$. The fibre $F$ of the natural projection to $\mathbb P^1$ is a divisor with self-intersection $0$; the exceptional fibre of the natural projection to $\mathbb P^2$ is a prime divisor $E$ with self-intersection $-2$. The line bundles associated to $E, F$ generate the Picard group $\Pic (F_2)$ with relations
\[
E^2 = -2, \quad E \cdot F = 1, \quad F^2 = 0.
\]
Now consider the derived category generated by the line bundles $\mathcal O_{F_2}$ and $\mathcal O_{F_2} (-E)$; we denote this category by $D^b (\reflectbox{$\mathrm{LG}$} (2))$, even though Theorem \ref{nopmirror} shows that it is not the derived category of coherent sheaves on any projective variety $\reflectbox{$\mathrm{LG}$} (2)$.

The $\Hom$ and $\Ext^k$ groups of line bundles on $F_2$ may be calculated via toric geometry, giving
\begin{equation}
\label{ext}
\begin{aligned}
\Hom (\mathcal O, \mathcal O) \simeq \Hom (\mathcal O (-E), \mathcal O (-E)) &\simeq \mathbb C \\
\Hom (\mathcal O (-E), \mathcal O) &\simeq \mathbb C \\
\Ext^1 (\mathcal O (-E), \mathcal O) \simeq H^1 (F_2, \mathcal O (E)) &\simeq \mathbb C
\end{aligned}
\end{equation}
all other $\Hom$ and $\Ext^k$ groups being zero.

Setting $\mathcal L_0 := \mathcal O_{F_2} (-E)$ and $\mathcal L_1 := \mathcal O_{F_2}$ we have in the derived category
\begin{align*}
\Hom (\mathcal L_i, \mathcal L_j) 
\simeq
\begin{cases}
\mathbb C \oplus \mathbb C [-1] & i < j \\
\mathbb C                       & i = j \\
0                               & i > j
\end{cases}
\end{align*}
in agreement with (\ref{morphismsfukaya}).
\end{proof}

\section{Quivers}
\label{quivers}

(\ref{ext}) shows that the collection $(\mathcal L_0, \mathcal L_1)$ has nonvanishing $\Ext^k$ groups for $k = 1$. Following \cite{hilleperling} we apply a ``partial mutation'' (modifying some elements in the collection) to find a pair of locally free sheaves generating $D^b (\reflectbox{$\mathrm{LG}$} (2))$, with {\it vanishing} $\Ext^k$ groups for $k > 0$.

Let $\mathcal E$ be a nontrivial extension of $\mathcal O (-E)$ by $\mathcal O$. We obtain a triangle
\[
\mathcal O \to \mathcal E \to \mathcal O (-E) \otimes \Ext^1 (\mathcal O (-E), \mathcal O) \to \mathcal O [1].
\]
By the results of \cite{hilleperling} (or by direct verification) we have:
\begin{lemma}
\label{tilting}
The pair $(\mathcal O, \mathcal E)$ generates $D^b (\reflectbox{$\mathrm{LG}$} (2))$ and has vanishing $\Ext^k$ groups for $k > 0$. In particular, $\mathcal O \oplus \mathcal E$ has no self-extensions and is a tilting bundle for $D^b (\reflectbox{$\mathrm{LG}$} (2))$.
\end{lemma}

The fact that the collections $(\mathcal O, \mathcal O (-E))$ and $(\mathcal O, \mathcal E)$ both generate $D^b (\reflectbox{$\mathrm{LG}$} (2))$ means that $D^b (\reflectbox{$\mathrm{LG}$} (2))$ is equivalent to the derived categories of (DG) modules over the corresponding (DG) quivers. These equivalences can be obtained by writing the objects of each collection as vertices and a basis for the morphisms between these objects as arrows.

The presence of nontrivial extensions means that the collection $(\mathcal O, \mathcal O (-E))$ gives rise to a {\it DG quiver} $\widetilde Q$
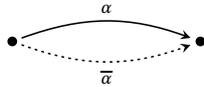
\begin{align*}
\begin{tikzpicture}[x=8em,y=1em]
\draw[line width=1pt, fill=black] (1,0) circle(0.3ex);
\draw[line width=1pt, fill=black] (0,0) circle(0.3ex);
\node[shape=circle, scale=0.9](L) at (0,0) {};
\node[shape=circle, scale=0.9](R) at (1,0) {};
\path[-stealth, line width=.6pt, line cap=round] (L) edge[out=20, in=160] node[above, font=\scriptsize] {$ \alpha$} (R);
\path[-stealth, line width=.6pt, line cap=round, dotted] (L) edge[out=-20, in=-160] node[below, font=\scriptsize] {$\overline \alpha$} (R);
\end{tikzpicture}
\end{align*}
where $\alpha$ is of degree $0$, $\overline \alpha$ of degree $1$. The associated path algebra $\widetilde A$ is a DG algebra with differential $\partial \alpha = \overline \alpha$. Since $\widetilde Q$ contains no composable arrows, all products except for multiplication by scalars vanish; {\it cf.}\ Lemma \ref{productsvanish}. We have an equivalence of triangulated categories $D^b (\text{dg mod-} \widetilde A) \simeq D^b (\reflectbox{$\mathrm{LG}$} (2))$. 

The absence of nontrivial extensions means that the collection $(\mathcal O, \mathcal E)$ gives rise to an ordinary quiver
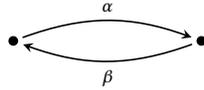
\begin{align*}
\begin{tikzpicture}[x=8em,y=1em]
\draw[line width=1pt, fill=black] (1,0) circle(0.3ex);
\draw[line width=1pt, fill=black] (0,0) circle(0.3ex);
\node[shape=circle, scale=0.9](L) at (0,0) {};
\node[shape=circle, scale=0.9](R) at (1,0) {};
\path[-stealth, line width=.6pt, line cap=round] (L) edge[out=20, in=160] node[above, font=\scriptsize] {$\alpha$} (R);
\path[stealth-, line width=.6pt, line cap=round] (L) edge[out=-20, in=-160] node[below, font=\scriptsize] {$\beta$} (R);
\end{tikzpicture}
\end{align*}
with relation $\beta \alpha = 0$. The associated path algebra $A$ is a noncommutative ordinary algebra, {\it i.e.}\ a DG algebra concentrated in degree $0$. Again, we obtain an equivalence of triangulated categories $D^b (\text{dg mod-}A) \simeq D^b (\text{mod-}A) \simeq D^b (\reflectbox{$\mathrm{LG}$} (2))$.

\end{document}